\documentclass[reqno]{amsart} \usepackage{amssymb} \voffset -1in
\usepackage{mathtools}

 \numberwithin{equation}{section}
\newtheorem{theorem}{Theorem}[section]
\newtheorem{lemma}{Lemma}[section]
\newtheorem{corollary}[lemma]{Corollary}
\newtheorem{proposition}[lemma]{Proposition}
\newcommand{\R}{\mbox{$\Bbb R$}}

 \newcommand{\g}{{\mathfrak g}}

\def\n{\mathfrak{n}}
\def\g{\mathfrak{g}}

\def\i{\mathfrak{i}}
\def\l{\mathfrak{l}}
\def\n{\mathfrak{n}}
\def\o{\mathfrak{o}}

\def\s{\mathfrak{s}}

\begin{document}

\title{Minimal representations of Lie algebras with non-trivial Levi decomposition} \maketitle

\begin{center}

 {\bf Ryad Ghanam$^{1}$, Manoj Lamichhane$^{2}$  and Gerard Thompson$^{3}$}\\

$^1$ Department of Mathematics, Virginia Commonwealth University in Qatar\\
PO Box 8095, Doha, Qatar\\
raghanam@vcu.edu\\
$^2$ Department of Mathematics, Western Kentucky University\\ Bowling Green, KY 42101\\manoj.lamichhane@wku.edu\\
$^3$ Department of Mathematics, University of Toledo,\\ Toledo, OH
43606, U.S.A. \\  gerard.thompson@utoledo.edu \\

\end{center}

\maketitle

\begin{abstract}
We obtain minimal dimension matrix representations for each of the Lie algebras of dimensions five, six, seven and eight obtained by Turkowski that have a non-trivial Levi decomposition. The key technique involves using the invariant subspaces associated to a particular representation of a semi-simple Lie algebra to help the in the construction of the radical in the putative Levi decomposition.
\end{abstract}

\vspace{4cm}

\begin{center}
\hspace{1cm} Key words: semi-simple Lie algebra, minimal representation, radical, Levi decomposition, invariant subspaces.

\hspace{1cm} AMS Mathematics Subject Classification (2010): 22E60 17B05, 17B30, 20G05, 22E25.
\end{center}

\newpage

\section{Introduction} Given a real Lie algebra $\g$ of dimension $n$ a
well known theorem due to Ado asserts that $\g$ has a faithful
representation as a subalgebra of $\g\l(p,\R)$ for some $p$. In several recently published papers the authors and others have investigated the problem of finding minimal
dimensional representations of indecomposable dimensional Lie algebras of dimension five and less \cite{TW,KB,GT1,GT2}.
In two recent papers \cite{KB,GT1}, the problem of finding such a
minimal representation is considered for the four-dimensional Lie
algebras. In \cite{KB}, the main technique is somewhat indirect and
depends on a construction known as a left symmetric structure. In
\cite{GT1, GT2}, the minimal representations have been calculated directly
without the need for considering left symmetric structures.

Burde \cite{Bur} has defined an invariant $\mu(\g)$ of a Lie algebra $\g$ to be the dimension of its minimal faithful representation. So far minimal representations have been found for all the indecomposable and decomposable algebras of dimension five and less \cite{GT1, GT2}. In this paper we investigate the issue of  minimal representations for Lie algebras that have a non-trivial Levi decomposition. Such algebras have been classified in dimensions five to eight by Turkowski \cite{Turk}.
Of course it is interesting to ascertain the value of $\mu$ from a theoretical point of view. However, an important practical reason is that calculations involving symbolic programs such as Maple and Mathematica use up lots of memory when storing matrices; accordingly, calculations are likely to be faster if one can represent matrix Lie algebras using matrices of a small size. Besides the value of having explicit representations of low-dimensional Lie algebras they also add to the growing body of results that seek to provide alternatives to Ado's Theorem for the construction of representations \cite{Bur}.\\
In this paper we shall use the classification of the low-dimensional \emph{indecomposable} Lie algebras found in \cite{PSWZ} which is in turn taken from \cite{Mub1}. Such algebras are denoted as $A_{m.n}$ where $m$ denotes the dimension of the algebra and $n$ the $n$th one in the list. In dimension $2$, the algebra $A_{2.1}$ is the unique non-abelian algebra. The same notation is adopted in \cite{Mub2}, which concerns a class of six-dimensional algebras, and is used in \cite{Turk}, which is the most important reference for the current paper. We shall also use $R_n$ to denote the $n\times n$ irreducible representation of $\s\o(3)$ and $D_j$ for the irreducible representation of order $(2j+1)\times (2j+1)$ of $\s\l(2,\R)$ where $j$ is either an integer of half-integer. Furthermore, $kD_j$ denotes $k$ copies of $D_j$ and $D_0$ the trivial one-dimensional representation.

The outline of this paper is as follows. In Section 2 we consider the construction of abstract Lie algebras that have a Levi decomposition. In Section 3 we present several results about representations that will needed be in the sequel. In Section 4 we carry out the construction of the minimal representations proper. A key element here is the use of the invariant subspaces that come from the representation of the semi-simple factor, in all cases here either $\s\o(3)$ or $\s\l(2,\R)$ in a given representation, that acts by commutator on $\g\l(N,\R)$ for various small values of $N$, to construct a representation of the radical ideal of the algebra. It is argued in every case that the representations thus constructed are of minimal dimension. Finally, in Section 5 we give a list of  minimal dimension representations for each of the algebras in Turkowski's list \cite{Turk}.
Occasionally we give more than one such representation, sometimes depending on parameters $\lambda$ and $\mu$ if it seems to be of particular interest although we do not consider the difficult issue of the inequivalence of different representations.
Finally, the reader may object that we have not supplied details of the Turkowski Lie algebras but to do so would involve repeating large Sections of \cite{Turk}. In any case the
Lie brackets can be readily computed from the representations given in Section 5.

\section{Constructing Levi algebras in general}

Let us consider the problem of constructing a Lie algebra $\g$ that
has a Levi decomposition $\sigma \rtimes \rho$ in general. We have
the following structure equations: \begin{equation} \label{rep}
[e_a,e_b]=C^c_{ab}e_c, \, [e_a,e_i]=C^k_{ai}e_k, \,
[e_i,e_j]=C^k_{ij}e_k \end{equation} \noindent where $1 \leq a,b,c,d
\leq s$ and $s+1 \leq i,j,k,l \leq n$ and $\{e_a\}$ is a basis for
the semi-simple subalgebra $\sigma$ and $\{e_i\}$ is a basis for the
radical $\rho$. Calculation shows that the Jacobi identity is
equivalent to the following conditions:

\begin{center} $C^e_{[ab}C^d_{c]e}=0,  \,\, C^c_{ab}C^j_{ci}=
C^k_{bi}C^j_{ak}-C^k_{ai}C^j_{bk},$ \end{center}

\begin{equation}
C^k_{al}C^l_{ij}=C^l_{ai}C^k_{lj}-C^l_{aj}C^k_{li}, \, \,
C^l_{[ij}C^m_{k]l}=0. \label{struc_eq} \end{equation}

We interpret these conditions as follows: we start with a
semi-simple algebra so that the first set of conditions above is
satisfied. Then the second set say that the matrices $C^k_{al}$ make
$\rho$ (merely as a vector space) into a $\sigma$-module. The third
set say that the $C^k_{al}$ are derviations of the Lie algebra
$\rho$ and the fourth of course that $\rho$ is a  Lie algebra.
Therefore to find all possible Lie algebras  of dimension $n$ that
have a Levi decomposition $\sigma \rtimes \rho$ we can proceed as
follows: choose a semi-simple algebra $\sigma$ of dimension $r$.
Then pick any solvable algebra $\rho$ of dimension $n-r$ and
consider a representation of $\sigma$ in $\rho$ considered simply as
the vector space $\R^{n-r}$, all of which are known and are
completely reducible \cite{Hum} since $\sigma$ is semi-simple.
Of course it is a difficult question as to whether a representation of a
particular semi-simple algebra exists in a certain dimension: see \cite{Hum} for more details.
Finally, it only remains to check that the matrices representing
$\sigma$ act as derivations of the Lie algebra $\rho$. In the
affirmative case we have our sought after Lie algebra; in the
negative case there is no such algebra and we have to choose a
different representation of $\sigma$ in $\rho$. If all such
representations lead to a null result then there can be no non-trivial Levi
decomposition involving $\sigma$ and $\rho$.

\section{Representation results}

We quote the following result from Humphreys \cite{Hum}.
\begin{proposition} \label{Hump}
Let $L$ be a finite dimensional complex Lie algebra acting irreducibly on the vector space $V$. Then $L$ is reductive and the center is of dimension one or zero. If in addition $L\subset \s\l(V)$ then $L$ is semi-simple.
\end{proposition}
\noindent Notice that the condition of semi-simplicity is a conclusion not an assumption. In practice we apply this result to semi-simple algebras to deduce that for a particular representation there is only a trivial one-dimensional extension to a Levi decomposition or more generally to an irreducible block in a reducible representation.

The following two results lie at opposite extremes.

\noindent \begin{theorem} \label{ab_rep_thm} Suppose that the
Lie algebra $\g$ is a semi-direct product of a semi-simple subalgebra $\sigma$ and  an
$r$-dimensional abelian ideal $\rho$ in the Levi decomposition. Then $\g$ has a (faithful)
representation as a subalgebra of $gl(r+1, \Bbb R)$.

Conversely, a subalgebra of $gl(r+1, \Bbb R)$ with the upper $r \times r$ block giving a representation of semi-simple Lie algebra and the first $r$ entries in the last column being arbitrary and the $(r+1)$th zero, gives a Levi subalgebra of $gl(r+1, \Bbb R)$ whose radical is abelian.
\end{theorem}

\begin{proof} We quote the structure equations in eqn
$\ref{rep}$ where now we assume that $C^k_{ij}=0$. To obtain
the required representation consider ad$(\sigma)$ restricted to
$\rho$, in other words the matrices $C^k_{ai}$. Define $E_a$ to be
the same matrix as $C^k_{ai}$ but now augmented by an extra bottom
row and last column of zeroes. Now define $E_k$ to be the $(r+1)
\times (r+1)$ matrix whose only non-zero is $1$ in the $(k,r+1)$th
position. Then $[E_a,E_i]=E_aE_i-E_iE_a$; now $E_iE_a=0$. As regards
$E_aE_i$ it is a matrix that has only non-zero entries in the last
column. That last column is $C^k_{ai}$ with $a$ and $k$ fixed,
augmented by a zero in the $(r+1,r+1)$th position which means that
$[E_a,E_i]=C^k_{ai}E_k$. Since $[E_a,E_b]=C^c_{ab}E_c$ and
$[E_i,E_j]=0$ we have a representation of $\g$ in $gl(r+1, \Bbb R)$. 
\end{proof}

\noindent \begin{theorem} \label{ab_rep_thm1} Suppose that $\rho$
is the radical of a Lie algebra $\g$ and that $\rho$ has trivial center. Then $\g$ has a
 faithful  representation of $\g$ in $\g\l(r,\Bbb R)$ where $\rho$ is of dimension $r$.
\end{theorem}

\begin{proof} Define $E_a$ to be the matrix $C^k_{ai}$ and $E_j$ to
be the matrix $C^k_{ij}$. Then in order to have a representation of
$\g$ we require that \begin{equation} \label{solveq1}
[E_a,E_b]=C^c_{ab}E_c, \, [E_a,E_i]=C^k_{ai}E_k, \,
[E_i,E_j]=C^k_{ij}E_k. \end{equation} However, these conditions are
identical with the last three in eqn. \ref{struc_eq}. If the center
of $\rho$ is trivial then we will have a faithful representation of
$\g$.
\end{proof}

The previous Theorem admits the following variation, having been already noted in the solvable case \cite{GT2}.

\noindent \begin{theorem} \label{ab_rep_thm2} Suppose that $\rho$
is the radical of a Lie algebra $\g$ and that $\rho$ has an abelian nilradical of dimension $r-1$ where $\rho$ is of dimension $r$. Then $\g$ has a
faithful  representation in $\g\l(r,\Bbb R)$.
\end{theorem}

\begin{proof} Define $E_a$ to be the matrix $C^k_{ai}$ and $E_r$ to
be the matrix $C^k_{ri}$. For $1\leq k \leq r-1$ define $E_k$ to be the $r
\times r$ matrix whose only non-zero entry is a $1$ in the $(k,r)$th
position. Then proceed as in Theorem \ref{ab_rep_thm1}.
\end{proof}


\section{Minimal Representations}

\subsection{Algebras of dimension $5,6,7$}

\subsubsection{Algebras of dimension $5$}

The first algebra that has a non-trivial Levi decomposition is denoted either by $A_{5.40}$ in \cite{PSWZ} or $L_{5.1}$ in \cite{Turk}. We obtain a representation for it in  $\g\l(3,\R)$ by using Theorem \ref{ab_rep_thm}.

\subsubsection{Algebras of dimension $6$}

In dimension six there are four algebras that have a non-trivial Levi decomposition denoted by $L_{6.1-6.4}$ in \cite{Turk}. According to \cite{TW} only one of them, $L_{6.3}$, has a  representation in $\g\l(3,\R)$ which is supplied by Theorem \ref{ab_rep_thm1}. Furthermore, Theorem \ref{ab_rep_thm} provides
representations of minimal dimension for $L_{6.1}$ and $L_{6.4}$ in $\g\l(4,\R)$. Furthermore $L_{6.2}$ has a representation in $\g\l(4,\R)$ which is also of minimal dimension.

\subsubsection{Algebras of dimension $7$}

According to \cite{TW} the only seven-dimensional algebra that has a representation in $\g\l(3,\R)$ is decomposable. There are seven algebras that have a non-trivial Levi decomposition denoted by $L_{7.1-7.7}$ in \cite{Turk}. Theorem \ref{ab_rep_thm1} provides
representations for $L_{7.1},L_{7.3},L_{7.4},L_{7.5}$ in $\g\l(4,\R)$ whereas Theorem \ref{ab_rep_thm} yields representations for $L_{7.2},L_{7.6}$ and $L_{7.7}$ in $\g\l(5,\R)$. In fact $L_{7.7}$ also has a representation in $\g\l(4,\R)$.

It remains to show that the representations for $L_{7.2},L_{7.6}$ coming from Theorem (\ref{ab_rep_thm}) are minimal. We have said already that they cannot be subalgebras of  $\g\l(3,\R)$. If they had representations in $\g\l(4,\R)$ there would be a basis in which the radical $\R^4$ would correspond to the upper right $2\times 2$ block according to a Theorem of Schur-Jacobson \cite{Jac}. As such the form of the representation would have to be $\left[\begin{smallmatrix}
  A & C  \\
  0 & B \\
\end{smallmatrix}\right]$. The occurrence of the lower left zero $2\times 2$ block is a consequence of the fact that the radical is an ideal. In the case of $L_{7.2}$ there is then no room to accommodate a representation of $\s\o(3)$; for $L_{7.6}$ we could only take either the standard  $2\times 2$ representation of $\s\l(2,\R)$ or the $4\times 4$ ``diagonal"  representation of $\s\l(2,\R)$. In the former case we obtain algebra $L_{7.7}$ and in the latter a decomposable subalgebra but certainly not $L_{7.6}$.




\section{Algebras of dimension $8$}

Now we consider Turkowski's algebras of dimension $8$. At the outset it is clear that for such algebras $\mu \geq 4$ because the only eight-dimensional subalgebra of $\g\l(3,\R)$ is $\s\l(3,\R)$, which is simple \cite{TW}. We shall find the algebras for which $\mu=4$.
\subsection{Algebras of dimension $8$ for which $\mu=4$}
According to  Proposition \ref{Hump} we cannot have the $4\times 4$ irreducible representation of $\s\l(2,\R)$ or $\s\o(3)$. Suppose next that we had the $3\times 3$ irreducible representation of $\s\l(2,\R)$ or $\s\o(3)$. A complement has to have matrices of the form
$\left[\begin{smallmatrix}
s & 0& 0&w \\
0 &s&0& x\\
0 & 0& s&y \\
p &q&r& z\\
\end{smallmatrix}\right]$. However, the  complement will be solvable only if either $w=x=y=0$ or $p=q=r=0$ in which case it gives an eight-dimensional \emph{decomposable} algebra. So the only possibilities, up to isomorphism, for the simple factor is either of the representations
$\left[\begin{smallmatrix}
a & b& 0&0 \\
c &-a&0& 0\\
0 & 0& 0&0 \\
0 &0&0& 0\\
\end{smallmatrix}\right]$ or $\left[\begin{smallmatrix}
a & b& 0&0 \\
c &-a&0& 0\\
0 & 0& a&b \\
0 &0&c& -a\\
\end{smallmatrix}\right]$.

In the first case the invariant subspaces may be described as
$\left[\begin{smallmatrix} 0_{2\times 2}&x&y \\
u^t &0&0\\
v^t &0&0\\
\end{smallmatrix}\right]$ where $u,v,x,y\in\R^2$ and each of them engenders a two-dimensional invariant subspace, together with a kernel
$\left[\begin{smallmatrix} \lambda I_2&0&0\\
0&d&e\\
0&f&g\\
\end{smallmatrix}\right]$. The radical cannot consist entirely of the kernel or else we will obtain a decomposable algebra, so one of $u,v,x,y$ must occur. Suppose that $y$ occurs. We will suppose also that $x$ occurs. We will need to add in a one-dimensional subspace of the kernel. It can be one of four kinds:

\small \begin{equation}\label{list}
\left[\begin{matrix}
\lambda h&0&0&0\\
0&\lambda h&0&0\\
0&0&\mu h&0\\
0&0&0&\nu h\\
\end{matrix}\right]
\quad
\left[\begin{matrix}
\lambda h&0&0&0\\
0&\lambda h&0&0\\
0&0&h&h\\
0&0&0&h\\
\end{matrix}\right]
\quad
\left[\begin{matrix}
\lambda h&0&0&0\\
0&\lambda h&0&0\\
0&0&\mu h&h\\
0&0&-h&\mu h
\end{matrix}\right]
\quad
\left[\begin{matrix}
0&0&0&0\\
0&0&0&0\\
0&0&0&h\\
0&0&0&0
\end{matrix}\right].
\end{equation}
It is not our concern here to look at all possible subalgebras of $\s\l(4,\R)$ (see \cite{GT3}); however, if we put $\nu=\lambda+p, \mu=1+\lambda$ in the first matrix in  \ref{list}, we obtain precisely $L_{8.17}$. In the second case we obtain $L_{8.16}$, in the third $L_{8.18}$ and in the fourth $L_{8.14}$.

Now suppose that $x=0$  but that $y\neq0$. Then another way to obtain an algebra is to take $u\neq0$ and independent of $y$. We will need to have $e\neq0$ in the kernel. We obtain thereby a radical that is isomorphic to the five-dimensional Heisenberg algebra: it is $L_{8.13 \epsilon=-1}$ in \cite{Turk}. Other choices of invariant subspaces will lead to different representations but not to different underlying algebras.

Finally we come back to the other representation $\left[\begin{smallmatrix}
a & b& 0&0 \\
c &-a&0& 0\\
0 & 0& a&b \\
0 &0&c& -a\\
\end{smallmatrix}\right]$ of $\s\l(2,\R)$. This time the invariant subspaces may be described as
$\left[\begin{smallmatrix}
A&B \\
C&D\\
\end{smallmatrix}\right]$, where each of $A,B,C,D$ is a copy of $\s\l(2,\R)$ together with the kernel
$\left[\begin{smallmatrix}
dI_2&eI_2\\
fI_2&gI_2\\
\end{smallmatrix}\right]$.
To avoid having a decomposable algebra we include $B$ and we need to take a two-dimensional kernel. Again to avoid decomposability, we shall have to take, up to equivalence, $e\neq0$ and $g=(\lambda+1)d$. We obtain $L_{8.12 p=1}$.

We know now that for all the remaining algebras $\mu>4$. Next we observe that Theorem \ref{ab_rep_thm1} is applicable to the following algebras:
$L_{8.1}, L_{8.3}, L_{8.4}, L_{8.7}, L_{8.8 p\neq0}, L_{8.9}, L_{8.10}, L_{8.11}, L_{8.12 p\neq0}, L_{8.20}$ and for them we have $\mu=5$. There remain nine cases which are
not yet covered: $L_{8.2}, L_{8.5}, L_{8.6},L_{8.8 p=0}, L_{8.13 \epsilon=1}, L_{8.15}, L_{8.19}, L_{8.21}, L_{8.22}$. We note that Theorem
\ref{ab_rep_thm} implies for $L_{8.5},L_{8.21},L_{8.22}$ that $\mu \leq 6$. In fact for $L_{8.22}$ we have $\mu = 5$. The cases of $L_{8.5}$ and $L_{8.21}$ will be discussed below.
As regards $L_{8.8}$ it depends on a parameter $p$; for $p\neq0$ the radical has trivial center but for $p=0$ there is a one-dimensional center. However, in the case $p=0$
we can find a representation by appealing to Theorem \ref{ab_rep_thm2}. This representation, unlike the adjoint representation of $\rho$, extends uniformly for all values of $p$.
Hence for $L_{8.8}$ we have $\mu=5$. The case $L_{8.15}$ will be discussed below. We shall also consider $L_{8.2}, L_{8.6}, L_{8.13 \epsilon= 1}, L_{8.19}$ where the radical is the five-dimensional Heisenberg algebra.
\subsection{$L_{8.5}$ and $L_{8.21}$} As stated above, Theorem \ref{ab_rep_thm} furnishes representations for $L_{8.5}$ and $L_{8.21}$ in dimension $6$. The question is whether
representations exist in lower dimensions. We know that the only possibility is a representation in dimension $5$. We examine $L_{8.5}$ first. We know that $\s\o(3)$ has a unique irreducible representation in every dimension starting at $3$. To find a representation in $\g\l(5,\R)$ we must consider the irreducible representations in dimensions $n=3,n=4$ and $n=5$. The case $n=5$ is impossible according to Proposition \ref{Hump}. For the case $n=3$ we argue as follows. Without loss of generality the representation for $\s\o(3)$ may be assumed to be of the form $\left[\begin{smallmatrix}
A&0_{3\times 2}\\
0_{2\times 3}&0_{2\times 2}
\end{smallmatrix}\right]$ where $A$ is the standard irreducible $3\times 3$ representation of $\s\o(3)$. Now consider the action of this representation of $\s\o(3)$ on
$\g\l(5,\R)$. We write such a matrix in $\g\l(5,\R)$ as
$\left[\begin{smallmatrix}
A+\alpha I_3&x&y\\
u^t&a&b\\
v^t&c&d\\
\end{smallmatrix}\right]$ where $u,v,x,y \in \R^3$ and $a,b,c,d,\alpha \in \R$. Now $u,v,x,y$ engender invariant subspaces and so too do the one-dimensional subspaces
corresponding to $a,b,c,d,\alpha$; the latter five span the kernel of the representation. In order to have a $5$-dimensional radical we must choose one of $u,v,x,y$ and a
$2$-dimensional subspace of the kernel. However, we will never obtain as the representation $R_5$ of $\s\o(3)$ on the radical, the irreducible $5\times 5$ representation of $\s\o(3)$. The argument for $n=4$ is similar if even easier.

The preceding argument lends itself to immediate generalization as follows. Whenever a $k\times k$ irreducible representation of a semi-simple algebra is augmented by adding a certain number of zero rows and columns and acts by commutator on $\g\l(N,\R)$ for some $N>k$, the corresponding invariant subspaces give a number of copies of the associated standard or definition representation, that is, a matrix acting by matrix multiplication on a vector, together with the kernel and a complement to the $k\times k$ irreducible representation inside of $\g\l(k,\R)$. In particular, it is impossible to obtain an $N\times N$ irreducible representation unless $k=N$ or from the complementary subspace; of course for $k=N$ this subspace is not available as it would change the given representation into $\g\l(N,\R)$.

Now we examine $L_{8.21}$ and apply the discussion of the last paragraph. The representation of $\s\l(2,\R)$ on the radical is supposed to give $D_2$. As such we can exclude the irreducible $4\times 4$ representation of sizes $2\times 2, 3\times 3$ and $4\times 4$. The $5\times 5$ representation is also not allowed because of Proposition \ref{Hump}. There remain two more representations to consider:

\[\small\begin{array}{llll}
\left[\begin{matrix}
2a&2b&0&0&0\\
c&0&b&0&0\\
0&2c&-2a&0&0\\
0&0&0&a&b \\
0&0&0&c&-a
\end{matrix}\right] \,\,
\left[\begin{matrix}
a&b&0&0&0\\
c&-a&0&0&0\\
0&0&a&b&0\\
0&0&c&-a&0 \\
0&0&0&0&0
\end{matrix}\right]
\end{array}.\]
In the first case the invariant subspaces can be described as:
\[\small \begin{array}{lllll}
\left[\begin{matrix}
a&b&0&0&0\\
c&-a&0&0&0\\
0&0&0&0&0\\
0&0&0&0&0\\
0&0&0&0&0
\end{matrix}\right]
\quad
\left[\begin{matrix}
0&0&0&0&0\\
0&0&0&0&0\\
0&0&2d&2e&0\\
0&0&f&0&e\\
0&0&0&2f&-2d
\end{matrix}\right]
\quad
\left[\begin{matrix}
0&0&0&0&0\\
0&0&0&0&0\\
0&0&g&-2h&i\\
0&0&j&-2g&h\\
0&0&k&-2j&g
\end{matrix}\right]
\quad
\quad
\left[\begin{matrix}
z&0&0&0&0\\
0&z&0&0&0\\
0&0&\alpha&0&0\\
0&0&0&\alpha&0\\
0&0&0&0&\alpha
\end{matrix}\right]
\end{array}\]
\[\small\begin{array}{lllll}
\left[\begin{matrix}
0&0&m&n&0\\
0&0&0&m&n\\
0&0&0&0&0\\
0&0&0&0&0\\
0&0&0&0&0
\end{matrix}\right]
\left[\begin{matrix}
0&0&p&-2q&r\\
0&0&s&-2p&q\\
0&0&0&0&0\\
0&0&0&0&0\\
0&0&0&0&0
\end{matrix}\right]
\quad
\left[\begin{matrix}
0&0&0&0&0\\
0&0&0&0&0\\
t&0&0&0&0\\
u&t&0&0&0\\
0&u&0&0&0
\end{matrix}\right]
\quad
\left[\begin{matrix}
0&0&0&0&0\\
0&0&0&0&0\\
v&x&0&0&0\\
-2w&-2v&0&0&0\\
y&w&0&0&0
\end{matrix}\right].
\end{array}\]
We see that none of these invariant subspaces are five-dimensional and hence we cannot obtain the representation $D_2$.

In the second case the invariant subspaces are given by:
\[\small  \begin{array}{llll}
\small\left[\begin{matrix}
A&0&0\\
0&0&0\\
0&0&0\\
\end{matrix}\right]

\left[\begin{matrix}
0&B&0\\
0&0&0\\
0&0&0\\
\end{matrix}\right]

\left[\begin{matrix}
0&0&0\\
C&0&0\\
0&0&0\\
\end{matrix}\right]

\left[\begin{matrix}
0&0&0\\
0&D&0\\
0&0&0\\
\end{matrix}\right]

\left[\begin{matrix}
aI_2&bI_2&0\\
cI_2&dI_2&0\\
0&0&\alpha\\
\end{matrix}\right]\\

\left[\begin{matrix}
0&0&u\\
0&0&0\\
0&0&0\\
\end{matrix}\right]

\left[\begin{matrix}
0&0&0\\
0&0&v\\
0&0&0\\
\end{matrix}\right]

\left[\begin{matrix}
0&0&0\\
0&0&0\\
x^t&0&0\\
\end{matrix}\right]

\left[\begin{matrix}
0&0&0\\
0&0&0\\
0&y^t&0\\
\end{matrix}\right]
\end{array}  \]
where $A,B,C,D$ are each copies of $\s\l(2,\R), u,v,x,y\in\R^3$ and $a,b,c,d,\alpha\in\R$. Again there is no possibility of obtaining the representation $D_2$.


\subsection{Radical $5$-dimensional Heisenberg}

\begin{proposition}
Of the algebras $L_{8.6},L_{8.13 \epsilon=\pm 1},L_{8.19}$ only $L_{8.13 \epsilon=-1}$ has a faithful representation in $\g\l(4,\R)$.
\end{proposition}
\begin{proof}
Each of the three algebras has a one-dimensional center spanned by $e_8$. In each case $e_8$ also spans the center of the Heisenberg algebra of dimension five which is the radical in these cases. In any matrix representation of Heisenberg, a non-zero center must be represented by a nilpotent matrix. Up to isomorphism, in $\g\l(4,\R)$ there are precisely three non-zero nilpotent matrices of rank three, two and one:

\[ \small (i)\begin{array}{llll}
\left[\begin{matrix}
0&1&0&0\\
0&0&1&0\\
0&0&0&1\\
0&0&0&0
\end{matrix}\right]\,
(ii)\,
\left[\begin{matrix}
0&0&1&0\\
0&0&0&1\\
0&0&0&0\\
0&0&0&0
\end{matrix}\right]
\, (iii)\,
\left[\begin{matrix}
0&0&0&1\\
0&0&0&0\\
0&0&0&0\\
0&0&0&0
\end{matrix}\right].
\end{array} \]
The centralizer of a matrix of type (i) is of the form $\left[\begin{smallmatrix}
a&b&c&d\\
0&a&b&c\\
0&0&a&b\\
0&0&0&a
\end{smallmatrix}\right]$, which is four-dimensional abelian. As regards a matrix of type (ii), its centralizer is of the form $\left[\begin{smallmatrix}
a&b&e&f\\
c&d&g&h\\
0&0&a&b\\
0&0&c&d
\end{smallmatrix}\right]$. The latter space of matrices comprise an eight-dimensional \emph{decomposable} Levi algebra, $\R^2\oplus L_{6.4}$ in fact. The conclusion is that if any of the algebras $L_{8.6},L_{8.13},L_{8.19}$ has a faithful representation in $\g\l(4,\R)$, a non-zero element in its center must be a multiple of a matrix of type (iii) in the appropriate basis.

The centralizer of a matrix of type (iii) is of the form
$\left[\begin{smallmatrix}
a&b&c&d\\
0&e&f&g\\
0&h&i&j\\
0&0&0&a
\end{smallmatrix}\right]$. This latter space of matrices comprise a decomposable ten-dimensional Levi algebra. In fact it is isomorphic to
$\R\oplus(\s\l(2,\R) \rtimes A_{6.82 \delta=0 a=b=1})$ \cite{Mub2}. Since the derived algebra of Heisenberg is spanned by the matrix of type (iii), the only way to obtain it as a subalgebra is in the obvious way, that is, by taking $a=0$ and $i=-e$: it gives the algebra $L_{8.13 \epsilon=-1}$.
\end{proof}

Now we shall investigate whether algebras that have a $5$-dimensional Heisenberg radical can be represented in $\g\l(2,\R)$. Again the central element $e_8$ must be represented by a nilpotent matrix $E_8$. The rank of $E_8$ cannot be $4$ otherwise its centralizer is abelian; nor can it be $3$ or else the centralizer is solvable. Suppose then that the rank of $E_8$ is $2$ and we may assume that
$E_8=\left[\begin{smallmatrix}
0&0&1&0&0\\
0&0&0&1&0\\
0&0&0&0&0\\
0&0&0&0&0\\
0&0&0&0&0
\end{smallmatrix}\right]$. Its centralizer may be written as $\left[\begin{smallmatrix}
A&B&x\\
0&A&0\\
0&y^t&\alpha
\end{smallmatrix}\right]$ where $A,B$ are $2\times 2$, $x,y\in \R^2$ and $\alpha\in\R$. We deduce first of all that this case cannot correspond to $L_{8.2}$ since we have to have a  representation of $\s\o(3)$.

Now we may assume that $A$ is the standard $2\times 2$ representation of $\s\l(2,\R)$ together with a multiple of the identity. In fact we have the ``$2D_{\frac{1}{2}}$" representation of the form
$\left[\begin{smallmatrix}
a&b&0&0&0\\
c&-a&0&0&0\\
0&0&a&b&0\\
0&0&c&-a&0\\
0&0&0&0&0
\end{smallmatrix}\right]$ and a putative representation of $5$-dimensional Heisenberg in the form
$\left[\begin{smallmatrix}
\lambda I_2&B&x\\
0&\lambda I_2&0\\
0&y^t&\alpha
\end{smallmatrix}\right]$. Now in the case of algebra $L_{8.6}$ the kernel of the representation is ``$3D_0$", that is, three-dimensional; however, this circumstance could arise only from having $\lambda$ and $\alpha$ both non-zero and independent in which case we would obtain a decomposable algebra. Hence $L_{8.6}$ is excluded. In the cases $L_{8.13}$ and $L_{8.19}$ the kernel is one-dimensional and is spanned by $E_8$. As such we may assume that $\lambda=\alpha=0$. Now consider invariant subspaces. Both $x$ and $y$ give invariant subspaces and so does the trace-free part of $B$. However, the only way to have an invariant subspace of dimension five, since the element $E_8$ must be included, is to take $B$ as  multiples of the identity, which, however does not produce a subalgebra. The conclusion is that if there is a $5\times 5$ representation of an algebra that has the $5$-dimensional Heisenberg as its radical, then the element $E_8$ must have rank one.

Now we assume that $E_8$ has rank one and is of the form
$E_8=\left[\begin{smallmatrix}
0&0&0&0&0\\
0&0&0&0&0\\
0&0&0&0&0\\
0&0&0&0&1\\
0&0&0&0&0
\end{smallmatrix}\right]$.  Its centralizer may be written as
$\left[\begin{smallmatrix}
a&b&c&0&x\\
d&e&f&0&y\\
g&h&i&0&z\\
u&v&w&s&t\\
0&0&0&0&s\\
\end{smallmatrix}\right]$ where $A,B$ are $2\times 2$, $x,y\in \R^2$ and $\alpha\in\R$. In order to obtain $L_{8.2}$ we would have to take the three-dimensional representation of
$\s\o(3)$ in the upper left $3\times 3$ block. However, in that case because $(x,y,z)$ and $(u,v,w)$ span invariant subspaces the only possibility is to obtain the \emph{seven}-dimensional Heisenberg algebra.

Considering the cases $L_{8.6},L_{8.13},L_{8.19}$ we see that the only possibility for the $3\times 3$ block is to have either the three-dimensional irreducible representation of
or two-dimensional irreducible representation augmented by an extra row and column of zeros of $\s\l(2,\R)$. In the former case we will obtain again the seven-dimensional Heisenberg algebra. In the latter case we have to consider the algebras that we are trying to represent: $L_{8.6},L_{8.13},L_{8.19}$. For $L_{8.13}$ and $L_{8.19}$ the kernel of the representation is one-dimensional, which leads to the self-same representation that we found above for $L_{8.13 \epsilon=-1}$ but now augmented by an extra middle row and column of zeroes. On the other hand for $L_{8.6}$ the representation is given by ``$D_{\frac{1}{2}}\oplus 3D_0$" whereas $(x,y)$ and $(u,v)$ will each give $D_{\frac{1}{2}}$ unless they are correlated, that is, $u=y, v=-x$ which gives $2D_{\frac{1}{2}}$. Thus $\mu>5$ in the cases $L_{8.6},L_{8.13 \epsilon=1},L_{8.19}$ and in fact for these three algebras $\mu=6$.

\subsection{$L_{8.15}$} Now let us look at $L_{8.15}$. First of all note that unlike the algebras just discussed, where the radical is five-dimensional Heisenberg, although
the radical is again nilpotent, $A_{5.3}$, the algebra as a whole has trivial center and so the adjoint representation is faithful. Thus $\mu\leq 8$.  Moreover, it was shown
in \cite{GT2} that $\mu=5$ for $A_{5.3}$ alone. Thus $\mu\geq 5$ for $L_{8.15}$. In fact in Section 5 the reader will see a $6\times 6$ representation. Thus all that remains
is to show that $\mu>5$ for $L_{8.15}$.

As in the preceding subsection we can argue that $E_8$ must be nilpotent. Again we only have to consider the case where $E_8$ has rank one or two.

\subsubsection{$E_8$ of rank $2$}

If it has rank two, as above,
its centralizer may be written as
$\left[\begin{smallmatrix}
A&B&x\\
0&A&0\\
0&y^t&\alpha
\end{smallmatrix}\right]$ where $A,B$ are $2\times 2$, $x,y\in \R^2$ and $\alpha\in\R$. Although $e_8$ is not a central element, its centralizer is spanned by $\{e_1,e_2,e_3,e_4,e_5,e_8\}$. Hence the simple factor is of the form
$\left[\begin{smallmatrix}
a&b&0&0&0\\
c&-a&0&0&0\\
0&0&a&b&0\\
0&0&c&-a&0\\
0&0&0&0&0
\end{smallmatrix}\right]$ where $a,b,c$ correspond to $E_1,E_2,E_3$, respectively, and
$E_8=\left[\begin{smallmatrix}
0&0&1&0&0\\
0&0&0&1&0\\
0&0&0&0&0\\
0&0&0&0&0\\
0&0&0&0&0
\end{smallmatrix}\right]$.
\noindent Imposing the conditions $[E_1,E_6]=E_1$ and $[E_2,E_6]=0$ implies that $E_6$ must be of the form
$E_6=\left[\begin{smallmatrix}
0&0&0&0&*\\
0&0&0&0&0\\
0&0&0&0&*\\
0&0&0&0&0\\
0&*&0&*&0
\end{smallmatrix}\right]$ and likewise from $[E_1,E_7]=-E_1$ and $[E_3,E_7]=0$ that
$E_7=\left[\begin{smallmatrix}
0&0&0&0&0\\
0&0&0&0&*\\
0&0&0&0&0\\
0&0&0&0&*\\
*&0&*&0&0
\end{smallmatrix}\right]$.

Next we impose the condition $[E_6,E_7]=E_8$ which implies that

\[\begin{array}{llll}
E_6=\left[\begin{smallmatrix}
0&0&0&0&d\\
0&0&0&0&0\\
0&0&0&0&0\\
0&0&0&0&0\\
0&0&0&e&0
\end{smallmatrix}\right]
E_7=\left[\begin{smallmatrix}
0&0&0&0&0\\
0&0&0&0&f\\
0&0&0&0&0\\
0&0&0&0&0\\
0&0&g&0&0
\end{smallmatrix}\right].
\end{array}\] However, now we find that $E_4=[E_6,E_8]$ and $E_5=[E_7,E_8]$ implies that $E_4$ and $E_5$ are zero. Hence no such representation can exist.

\subsubsection{$E_8$ of rank $1$} As above we assume that $E_8$ has rank one and is of the form
$E_8=\left[\begin{smallmatrix}
0&0&0&0&0\\
0&0&0&0&0\\
0&0&0&0&0\\
0&0&0&0&1\\
0&0&0&0&0
\end{smallmatrix}\right]$.  Its centralizer is
$\left[\begin{smallmatrix}
a&b&c&0&x\\
d&e&f&0&y\\
g&h&i&0&z\\
u&v&w&s&t\\
0&0&0&0&s\\
\end{smallmatrix}\right]$. There are two possibilities: either the upper left $3\times 3$ is the irreducible
representation of $\s\l(2,\R)$ or it is the $2\times 2$ representation augmented by an column of rows and zeros, which we take as the third row and third column. Consider the former case first so that
\[\begin{array}{llll}
E_1=\left[\begin{smallmatrix}
2&0&0&0&\\
0&0&0&0&0\\
0&0&-2&0&0\\
0&0&0&0&0\\
0&0&0&0&0
\end{smallmatrix}\right]
E_2=\left[\begin{smallmatrix}
0&2&0&0&0\\
0&0&1&0&0\\
0&0&0&0&0\\
0&0&0&0&0\\
0&0&0&0&0
\end{smallmatrix}\right]
E_3=\left[\begin{smallmatrix}
0&0&0&0&0\\
1&0&0&0&0\\
0&2&0&0&0\\
0&0&0&0&0\\
0&0&0&0&0
\end{smallmatrix}\right].
\end{array}\]
However, the condition $[E_1,E_6]=E_6$ implies that $E_6=0$. Hence there is no representation in this case.

Now we assume that
\[\begin{array}{llll}
E_1=\left[\begin{smallmatrix}
1&0&0&0&\\
0&-1&0&0&0\\
0&0&0&0&0\\
0&0&0&0&0\\
0&0&0&0&0
\end{smallmatrix}\right]
E_2=\left[\begin{smallmatrix}
0&1&0&0&0\\
0&0&0&0&0\\
0&0&0&0&0\\
0&0&0&0&0\\
0&0&0&0&0
\end{smallmatrix}\right]
E_3=\left[\begin{smallmatrix}
0&0&0&0&0\\
1&0&0&0&0\\
0&0&0&0&0\\
0&0&0&0&0\\
0&0&0&0&0
\end{smallmatrix}\right].
\end{array}\]

\noindent Then $[E_1,E_6]=E_6$ implies that
$E_6=\left[\begin{smallmatrix}
0&0&c&d&e\\
0&0&0&0&0\\
0&m&0&0&0\\
0&s&0&0&0\\
0&x&0&0&0
\end{smallmatrix}\right]$
$E_7=\left[\begin{smallmatrix}
0&0&0&0&0\\
0&0&h&i&j\\
k&0&0&0&0\\
r&0&0&0&0\\
w&0&0&0&0
\end{smallmatrix}\right]$ and we find that $[E_2,E_6]=0$ and  $[E_3,E_7]=0$ are identically satisfied.
Furthermore we find that
$E_4=[E_6,E_8]=\left[\begin{smallmatrix}
0&0&0&0&d\\
0&0&0&0&0\\
0&0&0&0&0\\
0&-x&0&0&0\\
0&0&0&0&0
\end{smallmatrix}\right]$
$E_5=[E_7,E_8]=\left[\begin{smallmatrix}
0&0&0&0&0\\
0&0&0&0&i\\
0&0&0&0&0\\
-w&0&0&0&0\\
0&0&0&0&0
\end{smallmatrix}\right]$. Now if $[E_2,E_5]=E_4$ then $i=d$ and $w=-x$ and $[E_3,E_4]=E_5$ is identically satisfied. For $[E_4,E_7]=0$ we obtain
$cx=dx=dk=ex+dr=0$. Assume first that $d=0$ then $c=0,e=0$. At this point we obtain a Lie algebra whose radical is $A_{5.1}$ and not $A_{5.3}$. If instead we assume $x=0$ then we have $k=r=0$. Again we obtain a Lie algebra whose radical is $A_{5.1}$ and not $A_{5.3}$.

\section{Minimal Representations of Turkowski's algebras}

\subsection{Dimension 5}

\subsubsection{$\s\l(2,\R) \rtimes \R^2: L_{5.1}$}

\[\small \begin{array}{llll}
\left[\begin{matrix}
a&b&d&\\
d&-a&e\\
0&0&0
\end{matrix}\right]
\end{array} \]

\subsection{Dimension 6}

\subsubsection{$\s\o(3) \rtimes \R^3: L_{6.1}$}

\[\small \begin{array}{llll}
& \left[\begin{matrix}
0&c&-b&d\\
-c&0&a&e\\
b&-a&0&f\\
0&0&0&0
\end{matrix}\right]
\end{array} \]

\subsubsection{$\s\l(2,\R) \rtimes A_{3.1 }: L_{6.2}$}

\[\small \begin{array}{llll}
 \left[\begin{matrix}
a&b&0&d\\
c&-a&0&e\\
e&-d&0&f\\
0&0&0&0
\end{matrix}\right]
\end{array} \]

\subsubsection{$(\lambda\neq \mu)\,  \s\l(2,\R) \rtimes A_{3.3 }: L_{6.3}$}

\[\small  \begin{array}{lll}
\left[\begin{matrix}
a+\lambda f &b&d\\
-c&\lambda f-a&e\\
0&0&(1+\lambda)f
\end{matrix}\right]
\end{array} \]

\subsubsection{$\s\l(2,\R) \rtimes \R^3: L_{6.4}$}

\[ (i)\small \begin{array}{llll}
 \left[\begin{matrix}
2a&2b&0&d\\
c&0&b&e\\
0&2c&-2a&f\\
0&0&0&0
\end{matrix}\right]

(ii) \left[\begin{matrix}
a&b&d&e\\
c&-a&f&-d\\
0&0&a&b\\
0&0&c&-a
\end{matrix}\right]
\end{array} \]

\subsection{Dimension 7}

\subsubsection{$(\lambda\neq\mu)\,\, \s\o(3) \rtimes A_{4.5 a=b=1}: L_{7.1}$}
\[ \small \begin{array}{llll}
& \left[\begin{matrix}
\lambda g&c&-b&d\\
-c&\lambda g&a&e\\
b&-a&\lambda g&f\\
0&0&0&\mu g
\end{matrix}\right]
\end{array} \]

\subsubsection{$\s\o(3) \rtimes \R^4: L_{7.2}$}
\[\small\begin{array}{llll}
\left[\begin{matrix}
0&-a&-b&-c&d\\
a&0&-c&b&e\\
b&c&0&-a&f\\
c&-b&a&0&g\\
0&0&0&0&0
\end{matrix}\right].
\end{array} \]

\subsubsection{$(\lambda\neq \mu,p=\frac{1}{\lambda-\mu}) \s\l(2,\R) \rtimes A_{4.5 a=1} = L_{7.3p}$}
\[\begin{array}{llll}
 & \left[\begin{matrix}
a+\lambda g&b&0&d\\
c&\lambda g-a&0&e\\
0&0&(\mu+1)g&f\\
0&0&0&\mu g
\end{matrix}\right]
\end{array} \]

\subsubsection{$(\lambda\neq \mu) \s\l(2,\R) \rtimes A_{4.9 b=1} = L_{7.4}$}
\[\begin{array}{llll}
 & \left[\begin{matrix}
a+\lambda g&b&0&d\\
c&\lambda g-a&0&e\\
-e&d&\mu g&f\\
0&0&0&(2\lambda-\mu)g
\end{matrix}\right]
\end{array} \]

\subsubsection{$\s\l(2,\R) \rtimes A_{4.5 a=b=1}: L_{7.5}$}

\[ \small (i)\, (\lambda\neq \mu)\, \begin{array}{llll}
\left[\begin{matrix}                                                                                                                                            2a+g\lambda&2c&0&f\\                                                                                                                                              b&g\lambda&c&e\\
0&2b&g\lambda-2a&d\\
0&0&0&g\mu
\end{matrix}\right]

\,(ii)\,  (\lambda\neq 1)\,
\left[\begin{matrix}
a+g&b&-e&d\\
c&g-a&f&e\\
0&0&a+\lambda g&b\\
0&0&c&\lambda g-a                                                                                                                                                                  \end{matrix}\right]
\end{array} \]

\subsubsection{$\s\l(2,\R) \rtimes \R^4: L_{7.6}$}
\[\small\begin{array}{llll}
 & \left[\begin{matrix}
3a&3b&0&0&d\\
c&a&2b&0&e\\
0&2c&-a&b&f\\
0&0&3c&-3a&g\\
0&0&0&0&0\\
\end{matrix}\right]
\end{array} \]

\subsubsection{$\s\l(2,\R) \rtimes \R^4: L_{7.7}$}
\[\begin{array}{llll}
 & \left[\begin{matrix}
a&b&d&f\\
c&-a&e&g\\
0&0&0&0\\
0&0&0&0
\end{matrix}\right]
\end{array} \]

\subsection{Dimension 8}

\subsubsection{$\s\o(3) \rtimes A_{5.7 a=1 b=1 c=p}: L_{8.1}$}

\[\small \left[ \begin {array}{cccccc}
h&c&-b&0&d\\
\noalign{\medskip}-c&h&a&0&e\\
\noalign{\medskip}b&-a&h&0&f\\
\noalign{\medskip}0&0&0&ph&pg\\
\noalign{\medskip}0&0&0&0&0
\end {array} \right]\]

\subsubsection{$\s\o(3) \rtimes A_{5.4}: L_{8.2}$}

\[\small \left[ \begin {array}{cccccc}
0&a&b&c&0&e\\\noalign{\medskip}-a&0&-c&
b&0&f\\\noalign{\medskip}-b&c&0&-a&0&d\\\noalign{\medskip}-c&-b&a&0&0&
g\\\noalign{\medskip}f&-e&g&-d&0&h\\\noalign{\medskip}0&0&0&0&0&0
\end {array} \right]\]

\subsubsection{$\s\o(3) \rtimes A_{5.7 a=b=c=1}: L_{8.3}$}

\[\small  \left[ \begin {array}{ccccc} 0&-\frac{b}{2}&-\frac{c}{2}&-\frac{a}{2}&d\\ \noalign{\medskip} \frac{b}{2}&0&-\frac{a}{2}&\frac{c}{2}&e\\
\noalign{\medskip}\frac{c}{2}&\frac{a}{2}&0&-\frac{b}{2}&f
\\ \noalign{\medskip}\frac{a}{2}&-\frac{c}{2}&\frac{b}{2}&0&g\\ \noalign{\medskip}0&0&0&0&h
\end {array} \right]\]

\subsubsection{$\s\o(3) \rtimes A_{5.17 a=1 b=1 c=p}: L_{8.4}$}

\[\small \left[ \begin {array}{cccccc}
ph&b&c+h&a&d\\
\noalign{\medskip}-b&ph&a&h-c&e\\
\noalign{\medskip}-c-h&-a&ph&b&f\\
\noalign{\medskip}-a&c-h&-b&ph&g\\
\noalign{\medskip}0&0&0&0&0
\end {array} \right]\]

\subsubsection{$\s\l(2,\R) \rtimes \R^5: L_{8.5}$}

\[\small \left[ \begin {array}{cccccc}
0&2c&4b&4a&0&d\\\noalign{\medskip}-2c&0&-4a&4b
&0&e\\\noalign{\medskip}-b&a&0&c&-6a&f\\\noalign{\medskip}-a&-b&-c&0&-6b&g
\\\noalign{\medskip}0&0&2a&2b&0&h\\\noalign{\medskip}0&0&0&0&0&0
\end {array} \right]\]

\subsubsection{$\s\l(2,\R) \rtimes A_{5.4}: L_{8.6}$}

\[\small \left[ \begin {array}{cccccc}
a&b&0&0&0&d\\\noalign{\medskip}c&-a&0&0
&0&e\\\noalign{\medskip}0&0&0&0&0&f\\\noalign{\medskip}0&0&0&0&0&g
\\\noalign{\medskip}e&-d&g&-f&0&h\\\noalign{\medskip}0&0&0&0&0&0
\end {array} \right]\]

\subsubsection{$\s\l(2,\R) \rtimes A_{5.7 pq\neq0}: L_{8.7}$}

\[\small \left[ \begin {array}{ccccc} h+a&b&0&0&d\\ \noalign{\medskip}c&h-a&0&0
&e\\ \noalign{\medskip}0&0&ph&0&f\\ \noalign{\medskip}0&0&0&
qh&g\\ \noalign{\medskip}0&0&0&0&0\end {array} \right]\]

\subsubsection{$\s\l(2,\R) \rtimes (A_{5.8 p=1},A_{5.9 \frac{1}{p}\frac{1}{p} p\neq0}): L_{8.8}$}

\[\small \left[ \begin {array}{ccccc} a+h&b&0&0&d\\ \noalign{\medskip}c&h-a&0&0
&e\\ \noalign{\medskip}0&0&ph&h&f\\ \noalign{\medskip}0&0&0&
ph&g\\ \noalign{\medskip}0&0&0&0&0\end {array} \right]\]

\subsubsection{$\s\l(2,\R) \rtimes (A_{5.13 a=1 q\neq0}): L_{8.9}$}

\[\small \left[ \begin {array}{ccccc} a+h&b&0&0&d\\ \noalign{\medskip}c&h-a&0&0
&e\\ \noalign{\medskip}0&0&ph&qh&f\\ \noalign{\medskip}0&0&-qh&
ph&g\\ \noalign{\medskip}0&0&0&0&0\end {array} \right]\]

\subsubsection{$\s\l(2,\R) \rtimes (A_{5.19 a=2\, p\neq0}): L_{8.10}$}

\[\small \left[\begin{array}{ccccc} a+h&b&0&0&d\\ \noalign{\medskip}c&h-a&0&0
&e\\ \noalign{\medskip}-e&d&2h&0&f\\ \noalign{\medskip}0&0&0&
ph&pg\\ \noalign{\medskip}0&0&0&0&0\end {array} \right]\]

\subsubsection{$\s\l(2,\R) \rtimes (A_{5.20 a=2\, p\neq0}): L_{8.11}$}

\[\small \left[\begin{array}{ccccc} a+h&b&0&0&d\\ \noalign{\medskip}c&h-a&0&0
&e\\ \noalign{\medskip}-e&d&2h&h&f\\ \noalign{\medskip}0&0&0&
2h&g\\ \noalign{\medskip}0&0&0&0&0\end {array} \right]\]

\subsubsection{$\s\l(2,\R)\rtimes A_{5.7 b=c}: L_{8.12 p=1}$}

\[\small\begin{array}{llll}
\left[\begin{matrix}
(\lambda+1)d+a&b&e&f\\
c&(\lambda+1)d-a&g&h\\
0&0&\lambda d+a&b\\
0&0&c&\lambda d-a\\
\end{matrix}\right].
\end{array} \]

\subsubsection{$\s\l(2,\R)\rtimes A_{5.7 b=c}: L_{8.12p}$}

\[\small\begin{array}{llll}
\left[\begin{matrix}
h+2a&2b&0&0&d\\
c&h&b&0&e\\
0&2c&h-2a&0&f\\
0&0&0&ph&g\\
0&0&0&0&0\\
\end{matrix}\right].
\end{array} \]

\subsubsection{$\s\l(2,\R) \rtimes A_{5.4}: L_{8.13 \epsilon=1}$}

\[\small \left[ \begin {array}{cccccc}
a&b&0&0&0&d\\\noalign{\medskip}c&-a&0&0
&0&e\\\noalign{\medskip}0&0&a&b&0&f\\\noalign{\medskip}0&0&c&-a&0&g
\\\noalign{\medskip}e&-d&g&-f&0&h\\\noalign{\medskip}0&0&0&0&0&0
\end {array} \right]\]

\subsubsection{$\s\l(2,\R) \rtimes A_{5.4}: L_{8.13 \epsilon=-1}$}

\[\small\begin{array}{llll}
\left[\begin{matrix}
a&b&0&d\\
c&-a&0&e\\
f&g&0&h\\
0&0&0&0
\end{matrix}\right]
\end{array}\]
\noindent To obtain the form used by Turkowski \cite{Turk} use:
\[\small\begin{array}{llll}
\left[\begin{matrix}
a&b&0&d+g\\
c&-a&0&e+f\\
2(f-e)&2(d-g)&0&4h\\
0&0&0&0
\end{matrix}\right]
\end{array}\]


\subsubsection{$\s\l(2,\R)\rtimes A_{5.1}=L_{8.14}$}

\[\begin{array}{llll}
\small & \left[\begin{matrix}
a&b&f&d\\
c&-a&g&e\\
0&0&0&h\\
0&0&0&0
\end{matrix}\right]
\end{array}\]

\subsubsection{$\s\l(2,\R)\rtimes A_{5.3}=L_{8.15}$}


\[ \left[ \begin {matrix}a&b&-h&0&d&e\\ \noalign{\medskip}c&-a
&0&-h&f&g\\ \noalign{\medskip}0&0&a&b&0&d\\ \noalign{\medskip}0&0&c&
-a&0&f\\ \noalign{\medskip}0&0&f&-d&0&2h\\ \noalign{\medskip}0&0&0&0&0&0
\end {matrix}\right]\]

\subsubsection{$\s\l(2,\R)\rtimes A_{5.15 a=1}: L_{8.16}$}

\[ \left[ \begin {array}{cccc} a&b&d&e\\ \noalign{\medskip}c&-a&f&g
\\ \noalign{\medskip}0&0&h&h\\ \noalign{\medskip}0&0&0&h\end {array}
 \right]\]

\subsubsection{$\s\l(2,\R)\rtimes A_{5.7 b=c}: L_{8.17} $}


\[\left[ \begin {array}{cccc} \lambda\,h+a&b&d&e\\ \noalign{\medskip}c&
\lambda\,h-a&f&g\\ \noalign{\medskip}0&0& \left( 1+\lambda \right) h&0
\\ \noalign{\medskip}0&0&0& \left( \lambda+p\right) h\end {array}
 \right]\]

\subsubsection{$\s\l(2,\R)\rtimes A_{5.17 p=q=s=1}: p=\mu-\lambda\, L_{8.18}$}

\[\begin{array}{llll}
\small & \left[\begin{matrix}
\lambda d+a&b&e&f\\
c&\lambda d-a&g&h\\
0&0&\mu d&d\\
0&0&-d&\mu d
\end{matrix}\right]
\end{array} \]

\subsubsection{$\s\l(2,\R)\rtimes \R^5:L_{8.19}$}

\[\small \left[ \begin{matrix} 3\,a&3\,b&0&0&0&{\it d}
\\ \noalign{\medskip}c&a&2\,b&0&0&e\\ \noalign{\medskip}0&2\,c&-a&b&0&
f\\ \noalign{\medskip}0&0&3\,c&-3\,a&0&g\\ \noalign{\medskip}-g&3\,f&-
3\,e&{\it d}&0&2\,h\\ \noalign{\medskip}0&0&0&0&0&0\end {matrix}
 \right]\]

\subsubsection{$\s\l(2,\R)\rtimes \R^5:L_{8.20}$}

\[\small \left[ \begin{matrix} h+3\,a&3\,b&0&0&{\it d}
\\ \noalign{\medskip}c&h+a&2\,b&0&e\\ \noalign{\medskip}0&2\,c&h-a&b&
f\\ \noalign{\medskip}0&0&3\,c&h-3\,a&g\\
\noalign{\medskip}0&0&0&0&0\end {matrix}\right]\]

\subsubsection{$\s\l(2,\R)\rtimes \R^5:L_{8.21}$}

\[\small\left[ \begin {matrix} 4a&4b&0&0&0&d\\ \noalign{\medskip}c&2a&3b&0&0&
e\\ \noalign{\medskip}0&2c&0&2b&0&f\\ \noalign{\medskip}0&0&3c&-2a&b&g
\\ \noalign{\medskip}0&0&0&4c&-4a&h\\0&0&0&0&0&0\end {matrix} \right]\]

\subsubsection{$\s\l(2,\R)\rtimes \R^5:L_{8.22}$}

\[\small\left[ \begin {matrix} a&b&d&e&g\\ \noalign{\medskip}c&-a&f&-d&
h\\ \noalign{\medskip}0&0&a&b&0\\ \noalign{\medskip}0&0&c&-a&0
\\ \noalign{\medskip}0&0&0&0&0\end {matrix} \right]\]

\section{Acknowledgment}

The authors thank the Qatar Foundation and Virginia Commonwealth University in Qatar for funding this project.

\end{document}